\documentclass[10pt,letterpaper]{amsart}
\usepackage{fullpage}
\usepackage{amsthm,amsfonts,amssymb,amsmath}
\usepackage{tabularx}
\usepackage{parskip}
\usepackage{enumitem}
\usepackage{comment}  
\usepackage[utf8]{inputenc}
\usepackage{xcolor}
\usepackage{amscd}
\usepackage{cancel}
\newtheorem{theorem}{Theorem}
\newtheorem{lemma}[theorem]{Lemma}
\newtheorem{cor}[theorem]{Corollary}
\newtheorem{prop}[theorem]{Proposition}
\theoremstyle{definition}
\newtheorem{df}[theorem]{Definition}
\newtheorem{rem}[theorem]{Remark}
\newtheorem{ex}[theorem]{Example}

\newtheorem{conj}[theorem]{Conjecture}

\DeclareMathOperator{\ch}{\mathrm{ch}}

\definecolor{cadmiumorange}{rgb}{0.93, 0.53, 0.18}
	\definecolor{darkorange}{rgb}{1.0, 0.35, 0.0}

\date{\today}
\title{Ramanujan sums and  rectangular power sums}
\author{John Shareshian and Sheila Sundaram}
\address{
Department of Mathematics,
Washington University, 
St. Louis, MO 63130}
\email{jshareshian@wustl.edu}
\address{School of Mathematics, University of Minnesota, Minneapolis, MN 55455}
\email{shsund@umn.edu}
\begin{document}

\begin{abstract} For a fixed nonnegative integer $u$ and positive integer $n$, we investigate the symmetric function 
\[\sum_{d|n} \left(c_d(\tfrac{n}{d})\right)^u  p_d^{\tfrac{n}{d}},\]
  where $p_n$ denotes the $n$th power sum symmetric function, and $c_d(r)$ is a \emph{Ramanujan sum}, equal to the sum of the $r$th powers of all the primitive $d$th roots of unity.   We establish the Schur positivity of these functions for $u=0$ and $u=1$, showing that, in each case, the associated representation of the symmetric group $\mathfrak{S}_n$ decomposes into a sum of Foulkes representations, that is, representations induced from the irreducibles of the cyclic subgroup generated by the long cycle.   We also   conjecture Schur positivity for the case $u= 2$.
\end{abstract}

\maketitle

\section{Introduction} 

Let $p_n$ be the $n$th power sum symmetric function. 
This paper is motivated by a question of Fr\'ed\'eric Chapoton, who asked in \cite{Chapoton2021}
 whether the sum $\sum_{d|n} p_d^{\tfrac{n}{d}}$ is Schur positive.  Equivalently, one wants to know if this sum is the Frobenius characteristic of a true representation of the symmetric group $\mathfrak{S}_n$ on $n$ letters.

We answer this question in the affirmative in Theorem~\ref{thm:Chapoton-Shareshian}, thereby augmenting the list of known Schur positive  multiplicity-free sums of power sums, see \cite[Theorems~3.6, 3.11, 3.12]{SuJCTA2021} and \cite[pp.40-41]{SuElecJComb2019}. We also establish the Schur positivity of the following variation:
\[\sum_{d|n} c_d(\tfrac{n}{d}) p_d^{\tfrac{n}{d}}\]
where $c_d(r)$ is a \emph{Ramanujan sum}, defined to be equal to the sum of the $r$th powers of all the primitive $d$th roots of unity  \cite{Ramanujan1918}, \cite{HW1979}, \cite{Knopf1975}, \cite{McCarthy1986}.  We do this by showing that the associated representation of $\mathfrak{S}_n$ decomposes into a direct sum of copies of certain well-studied representations, namely the irreducible representations of   the cyclic subgroup generated by the long cycle $(12\cdots n)$, induced up to $\mathfrak{S}_n$.  We call the latter representations the \emph{Foulkes representations}, see Theorem~\ref{thm:Foulkes} below.

Our main tool is a result, see Theorem~\ref{thm:FowlerGarciaKaraali}, from a paper of Fowler, Garcia and Karaali \cite{FGK2014}, which studies  Ramanujan sums as supercharacters of the cyclic group. Many identities satisfied by the Ramanujan sums, both new and classical, are derived in \cite{FGK2014} using the  framework of supercharacters. 

Section~\ref{sec:background} collects the preliminary facts about arithmetic functions  from  \cite{FGK2014} and the literature,   as well as background information on symmetric functions. 

\section{Background}\label{sec:background}
\subsection{Some classical arithmetic functions}

Let $\phi(d)$ be Euler's totient function,  $\mu(d)$ the number-theoretic M\"obius function, and $\tau(d)$ the divisor function, the number of divisors of $d$.  Also write $(d,r)$ for the greatest common divisor of $d$ and $r$. 

\begin{df}  For  positive  integers $r$ and  $d$, define the Ramanujan sum $c_d(r)$ to be the sum of the $r$th powers of all the primitive $d$th roots of unity, 
\[c_d(r):=\sum_{1\le m\le d: (m,d)=1} \exp(\tfrac{2imr\pi}{d})\]
\end{df}

Ramanujan sums are integers, see e.g. \cite{HW1979}, and have various equivalent formulas:

\[c_d({r})=\sum_{\substack{m\\m|(d,{r})}} m\mu(\tfrac{d}{m}) = \sum_{\substack{m=1\\(m,d)=1}}^{d} \exp({r}\tfrac{2im\pi}{d}) = \sum_{\substack{m=1\\(m,d)=1}}^{d} \cos({r}\tfrac{2im\pi}{d}) \]

The \emph{von Sterneck function} is defined to be the integer $\phi(d) \tfrac{\mu(\frac{d}{(d, r)})} {\phi(\frac{d}{(d, r)})}$, see \cite{HW1979}, \cite{FGK2014}.

There is one more expression for the Ramanujan sum that is particularly relevant for us.  
The following is often attributed to H\"older, though an equivalent identity was given earlier by Kluyver \cite[p. 410]{Kluyver1906}. See \cite[p. 221]{FGK2014} for historical discussion, and also \cite{HW1979}.

\begin{theorem}\label{thm:Kluyver1906}  The Ramanujan sum is an integer, equal to the von Sterneck function:
 \[c_d(r)= \dfrac{\phi(d)}{\phi(\tfrac{d}{(d, r)})}  \mu(\tfrac{d}{(d, r)}).\]

\end{theorem}

In particular, we have the well-known special cases 
\begin{equation}\label{eqn:SpecCasephimu}
c_d(r)
=\begin{cases} \mu(d), \quad (d,r)=1,\\
                         \phi(d), \quad d|r.
 \end{cases}
\end{equation}

\begin{cor}\label{cor:Factorisation} (See \cite[Theorem~3.3]{FGK2014})
If $m,n,x,y$ are positive integers with 
$mx, ny$ relatively prime, then \[c_{mn}(xy)=c_m(x) c_n(y).\]
\end{cor}

Direct computation from Theorem~\ref{thm:Kluyver1906} gives 
\begin{cor}\label{cor:prime-powers} Let $q$ be prime,  $a, n$ nonnegative integers.  Then 
\[c_{q^a}({n})=\begin{cases}   \phi(q^a), & q^a|n,\\
                                                   -q^{a-1}, & q^{a-1} |n, q^a\!\!\not|\, n,\\
                                                     0, &\text{otherwise} .
\end{cases}\]
\end{cor}
We observe the following about the diagonal Ramanujan sum $c_d(\tfrac{n}{d})$.
\begin{prop}\label{prop:Ramsum0pm1} Fix $n\ge 1$. The Ramanujan sum  $c_d(\tfrac{n}{d})$  takes on the values $\{0, \pm 1\}$ for every divisor $d$ of $n$ if and only if $n$ is square-free or 4 times an odd square-free integer. Furthermore, when $n$ is square-free,  
$c_d(\tfrac{n}{d})=\mu(d)=\pm 1$ for every divisor $d$ of $n$.  When $n$ is 4 times an odd square-free integer,  
 $|c_d(\tfrac{n}{d})|=1$ for every square-free divisor of $n$; it is zero if $4|d$.
\end{prop}
\begin{proof}  If $n$ is square-free,  $(d, \tfrac{n}{d})=1$ for every divisor $d$ of $n$ and so 
$c_d(\tfrac{n}{d})=\mu(d)=\pm 1$. 

Now suppose $n=4 k$ where $k$ is an odd square-free integer. Let $d$ be a divisor of $n$.   We use Corollary~\ref{cor:Factorisation}.  There are three possibilities for $d$.

If $d$ is odd, then $d|k$ and we have $c_d(\tfrac{n}{d})=c_1(4) c_d(\tfrac{k}{d}) =\mu(d)=\pm 1,$ since $d$ is square-free.

If $d=2 d'$ for odd $d'$, then $c_d(\tfrac{n}{d})=c_2(2) c_{d'}(\tfrac{k}{d'}) =\phi(2)\mu(d')=\pm 1,$ similarly.  

If $d=4 d'$ for odd $d'$, then $c_d(\tfrac{n}{d})=c_4(1)c_{d'}(\tfrac{k}{d'}) =\mu(4)=0$. 

On the other hand, suppose $n=d^2 k$ has a perfect square divisor $d^2$, $d>1$.  Then  $c_d(dk)=\phi(d)$, and this is greater than 1 if  $d>2$, and equal to 1 if $d=2$. 
It remains to consider the case  $n=4 k$ where $k$ is square-free.  If $k$ is even,  say $k=2k'$ where $k'$ is odd and square-free, then by Corollary~\ref{cor:Factorisation} and Theorem~\ref{thm:Kluyver1906}, 
$$c_4(\tfrac{n}{4})=c_4(k)=c_4(2) c_1(k')=\frac{\phi(4)}{\phi(2)}\mu(2)=-2.$$

We have shown that $|c_d(\tfrac{n}{d})|>1$ if $n$ has a perfect square divisor greater than 4, or if $n$ is 8 times an odd square-free integer.  
This completes the proof.
\end{proof}
\begin{prop}\label{prop:Ramsum-max} Fix $n\ge 1$. Then $|c_d(r)|\le \phi(d)$ with equality if and only if one of the following holds:
\begin{enumerate} \item $r=k d$ for some integer $k\ge 1$ or \item $d$ is even and $r= (2k-1)\frac{d}{2}$ for some integer $k\ge 1$. \end{enumerate}

\begin{proof} From Theorem~\ref{thm:Kluyver1906} the first inequality is clear.  To see when the maximum is attained, note that we must have 
$\phi(\tfrac{d}{(d, r)})=1$, which implies $\tfrac{d}{(d, r)}=1$ or $\tfrac{d}{(d, r)}=2$, that is, $(d, r)=d$ or $(d, r)=\frac{d}{2}, \text{ with $d$ even}.$ Since $\mu(\tfrac{d}{(d, r)})\ne 0$ in each case, the result follows.
\end{proof}
\end{prop}

Let $\delta_{a,b}$ denote the Kronecker delta; it is nonzero if and only if $a=b,$ in which case it equals 1.
Let $\{d_1, d_2,\ldots , d_{\tau(n)}\}$ be any ordering of the $\tau(n)$ divisors of $n$, and 
let $M_n$ be the matrix whose $(i,j)$ entry is $c_{d_i}(\frac{n}{d_j}).$  We refer to this matrix as the \emph{Ramanujan matrix}. It will play an important role in this paper.   

The following fact is classical, and is rederived using supercharacter theory in \cite[Theorem~3.5, Theorem~3.11]{FGK2014}.  
\begin{prop}\label{prop:Orthogonality1}\cite[Lemma~7.2.2]{Knopf1975} For divisors $d_1, d_2$ of $n$, 
\begin{equation}\frac{1}{n}\sum_{k|n} c_k(\tfrac{n}{d_1}) c_{d_2}(\tfrac{n}{k})=\delta_{d_1, d_2}.\end{equation}
Equivalently, 
\begin{equation}M_n^2=n I_n,\end{equation}
where $I_{n}$ is the $\tau(n)$ by $\tau(n)$ identity matrix.
It follows that the matrix $M_n$ is invertible, and the $(i,j)$-entry of $M_n^{-1}$ is $\frac{1}{n} c_{d_i}(\frac{n}{d_j}).$
\end{prop}
Since $\mu(k)=c_k(1)$, the case $d_1=n$ in the orthogonality relation of Proposition~\ref{prop:Orthogonality1} immediately gives:
\begin{prop}\label{prop:KeyFact} \cite{HW1979}, \cite[Corollary 3.13]{FGK2014} If $n\ge 1$ and $d|n,$ then 
\[\sum_{k|n} c_d(\tfrac{n}{k}) \mu(k) = n\delta_{d,n}.\]
\end{prop}
In the next two propositions we record two more identities for the diagonal Ramanujan sums $c_d(\tfrac{n}{d})$. 
\begin{prop}\label{prop:trM_n}(See \cite[Corollary~3.6]{FGK2014})
Let $n\ge 1.$ Then the trace of the Ramanujan matrix $M_n$ is given by 

\[\sum_{d|n}c_d(\tfrac{n}{d})
=\begin{cases} \sqrt{n}, \text{ if $n$ is a perfect square}, \\
                                                                       0, \text{ otherwise}. \end{cases}\]

\end{prop}

We will need the following additional computation in Section~\ref{sec:weightedRectangular}.  The identity below may be viewed as a formula for the \emph{signed} trace of the matrix $M_n$. A similar identity appears in \cite[Ex. 2.7, p. 90]{McCarthy1986}.
\begin{prop}\label{prop:signed-trM_n} Let $n\ge 1$ be even. Then 
\[\sum_{d|n}c_d(\tfrac{n}{d}) (-1)^{\tfrac{n}{d}}
=\begin{cases} \sqrt{n}, &\text{ if $n$ is a perfect square}, \\
                        2  \sqrt{n/2},  &\text{ if $n/2$ is an odd perfect square},\\
                          0,               &\text{ otherwise}. \end{cases}\]
\end{prop} 

\begin{proof}
Note that we have, for any $n$, 
\[\sum_{d|n}c_d(\tfrac{n}{d})=\sum_{\substack{d|n\\ \tfrac{n}{d} \text{ even}}} c_d(\tfrac{n}{d}) + \sum_{\substack{d|n\\ \tfrac{n}{d} \text{ odd}}} c_d(\tfrac{n}{d}),\]
and 
\[\sum_{d|n}c_d(\tfrac{n}{d}) (-1)^{\frac{n}{d}}=\sum_{\substack{d|n\\ \tfrac{n}{d} \text{ even}}} c_d(\tfrac{n}{d}) - \sum_{\substack{d|n\\ \tfrac{n}{d} \text{ odd}}} c_d(\tfrac{n}{d}).\]
It follows that, for arbitrary $n$, 
\begin{equation}\label{eqn:signed-Ramsum}
\sum_{d|n}c_d(\tfrac{n}{d}) (-1)^{\frac{n}{d}}- \sum_{d|n}c_d(\tfrac{n}{d})
= -2  \sum_{\substack{d|n\\ \tfrac{n}{d} \text{ odd}}} c_d(\tfrac{n}{d}).
\end{equation}

Now assume $n=2^a b$ is even, with $b$ odd, so that $a\ge1$.

We have
\begin{eqnarray*}
\sum_{d|n}(-1)^{n/d}c_d(\tfrac{n}{d}) &=  &\sum_{d|n}c_d(\tfrac{n}{d})-2\sum_{2^a|d|n}c_d(\tfrac{n}{d}) \\ 
&= & \left\{ \begin{array}{ll} \sqrt{n}-2\sum_{2^a|d|n}c_d(\tfrac{n}{d}), & n \mbox{ a perfect square}, \\ -2\sum_{2^a|d|n}c_d(\tfrac{n}{d}) &\mbox{otherwise,} \end{array} \right.
\end{eqnarray*}
the last equality following from Proposition~\ref{prop:trM_n}.

Say $2^a|d|n$.  So, $d=2^af$ for some divisor $f$ of $b$.  Now $ \phi(d)=2^{a-1}\phi(f)$
by the multiplicativity of $\phi$, and $ (d,\tfrac{n}{d})=(f,\tfrac{b}{f})$ 
since $\tfrac{n}{d}$ is odd.   Hence
$$
\phi\left(\frac{d}{(d,\tfrac{n}{d})}\right)=\phi\left(2^a\frac{f}{(f,\tfrac{b}{f})}\right)=2^{a-1}\phi \left(\frac{f}{(f,\tfrac{b}{f})}\right),
$$ 
and
$$
\mu\left(\frac{d}{(d,\tfrac{n}{d})}\right)=\mu\left(2^a\frac{f}{(f,\tfrac{b}{f})}\right)=\left\{ \begin{array}{ll} -\mu\left(\frac{f}{(f,\tfrac{b}{f})}\right), & a=1, \\ 0 & \mbox{otherwise}, \end{array} \right.
$$ 
by the definition of the M\"obius function.  Using Proposition~\ref{prop:trM_n} again, we conclude that if $a=1$ then
\begin{eqnarray*}
\sum_{2^a|d|n}c_d(\tfrac{n}{d}) & = & -\sum_{f|b}\frac{\phi(f)}{\phi(\frac{f}{(f,\tfrac{b}{f})})}\mu(\frac{f}{(f,\tfrac{b}{f})}) \\ & =  & -\sum_{f|b}c_f(\tfrac{b}{f}) \\ & = & \left\{ \begin{array}{ll} -\sqrt{b}, & b=n/2 \mbox{ a perfect square}, \\ 0 & \mbox{otherwise}. \end{array} \right.
\end{eqnarray*}
On the other hand, if $a>1$ then 
$\sum_{2^a|d|n}c_d(\tfrac{n}{d})=0, $ which completes the proof. 
\end{proof}

The following special case of one of the main results of the paper \cite{FGK2014} gives an explicit formula for the row sums of the Ramanujan matrix $M_n$.
\begin{theorem}\label{thm:FowlerGarciaKaraali} \cite[Theorem 4.5, the case $s=1$]{FGK2014} Let $n=\prod_{i=1}^r p_i^{\alpha_i}$ be the factorisation of $n$ into powers of distinct primes, and let $d=\prod_{i=1}^r p_i^{\beta_i}$ be a divisor of $n,$ so that $0\le \beta_i\le \alpha_i,$  and $\alpha_i\ge 1$ for all $i$. Then
\begin{equation}\label{eqn:FGK-rowsum}\sum_{k|n} c_d(k) = \prod_{\ell=1}^r \left ((\alpha_\ell-\beta_\ell+1) \phi(p_\ell^{\beta_\ell}) 
- \lfloor p_\ell^{\beta_\ell-1}\rfloor\right).
\end{equation}
\end{theorem}

From this we deduce the following observation, which will be crucial to the proof of Theorem~\ref{thm:Chapoton-Shareshian}.  

\begin{df}\label{def:row-sum-Mn} For $n\ge 1$, denote by $a_d(n)$ the row sum of $M_n$ indexed by the divisor $d$ of $n$.  Thus $a_d(n)=\sum_{k|n} c_d(k)$.
\end{df}

\begin{cor}\label{cor:Ramrowsum-nonneg} The row sum $a_d(n)$ of $M_n$ indexed by the divisor $d$ of $n,$  is always  nonnegative.  Moreover, $a_d(n)=0$ if and only if $n$ is even and $\frac{n}{d}$ is odd. In particular $a_n(n)=0$  for all even $n$.
\end{cor}
\begin{proof} For a prime $q$ and $n$, $d$ such that $d|n$ and $q|n$, write  $f(n,d, q)$ 
for the expression 
$(\alpha-\beta+1)\phi(q^{\beta}) -  \lfloor q^{\beta-1}\rfloor$
where $n=q^\alpha m_1$, $\alpha\ge 1$,  and  $d=q^\beta m_2$, $0\le \beta\le \alpha$ 
and $q\!\not |\, m_i,$ $i=1,2$. 

Then~\eqref{eqn:FGK-rowsum} says that 
\[a_d(n)=\sum_{k|n} c_d(k) = \prod_{\substack{q|n\\ q\text{ prime}}} f(n,d,q) .\]
We observe that if $\beta=0$ then
$$
f(n,d,q)  = \alpha+1,
$$
while if $\beta>0$ then
\begin{eqnarray*}
f(n,d,q) & = & (\alpha-\beta+1)q^{\beta-1}(q-1)-q^{\beta-1} \\ & = & q^{\beta-1}((\alpha-\beta)(q-1)+q-2).
\end{eqnarray*}
It follows that $f(n,d,q) \geq 0$, with equality if and only if $q=2$ and $\beta=\alpha$, the second condition holding if and only if $n/d$ is odd. 
\end{proof}

\subsection{The Foulkes representations of the symmetric group $\mathfrak{S}_n$}\label{sec:Foulkes-reps} Denote by $C_n$ the cyclic subgroup of the symmetric group $\mathfrak{S}_n$ generated by the long cycle $(1,2,\ldots,n). $   We define a Foulkes representation of $\mathfrak{S}_n$ to be any of the $n$ representations  obtained by inducing an irreducible representation of   $C_n$   up to $\mathfrak{S}_n$.  
For $1 \leq r \le n$, let $\ell_n^{(r)}$ denote the Frobenius characteristic of the induced 
representation $\mbox{exp}(\frac{2i\pi}{n} \cdot r)\big\uparrow_{C_n}^{\mathfrak{S}_n}$.  We observe that two elements of
 $C_n$ are conjugate in $\mathfrak{S}_n$ if and only if they have the same order.  
It follows that $\ell_n^{(r)}=\ell_n^{(s)}$ if and only if $(n,r)=(n,s)$.  
In particular, each $\ell_n^{(r)}$ is equal to a unique $\ell_n^{(d)}$ with $d|n$, 
hence we have $\tau(n)$ different Foulkes characters.  The symmetric function $\ell_n^{(1)}$, also denoted by $Lie_n$ in \cite{SuJCTA2021}, is the Frobenius characteristic of the well-studied representation  of $\mathfrak{S}_n$ on the multilinear component of the free Lie algebra on $n$ generators, see \cite{Reutenauer} and  \cite[Chapter 7, Ex. 7.89]{RPSEC2}.

We refer the reader to \cite{Macdonald2015-1995} and \cite[Chapter 7]{RPSEC2} for the basics of symmetric functions and the character theory of the symmetric group. In what follows $p_n$ denotes the power sum symmetric function of homogeneous degree $n$, and $\langle\ , \ \rangle$ is the Hall inner product in the ring of symmetric functions.

Now we state the theorem of Foulkes,   
which 
asserts Part (1) of the following \cite{Foulkes1972}. See  \cite [Ex. 7.88]{RPSEC2} for the definition of the major index statistic on tableaux. 
\begin{theorem}\label{thm:Foulkes} Let $1\leq r  \leq n.$  
\begin{enumerate}
\item (Foulkes \cite{Foulkes1972}, \cite[Ex. 7.88]{RPSEC2})
$\ell_n^{(r)}=\dfrac{1}{n} \sum_{d|n} c_d(r) p_d^\frac{n}{d}.$
\item (\cite[(7.191)]{KW2001},  \cite[Ex. 7.88]{RPSEC2}) The multiplicity $\langle \ell_n^{(r)}, s_\lambda\rangle $ of the Schur function $s_\lambda$ in  $\ell_n^{(r)}$ is the number of standard Young tableaux of shape $\lambda$ with major index congruent to $r$ modulo $n.$
\end{enumerate}
\end{theorem}

The last theorem that we record here for later use is a result of Swanson \cite{Swanson2018}.  It states that almost all irreducibles occur in the representations $\ell_n^{(r)}$.

\begin{theorem}\label{thm:Swanson}\cite[Theorem~1.5]{Swanson2018}
The multiplicity of the irreducible indexed by
 $\lambda$ in $\ell_n^{(r)}$, $1\le r\le n$,  is zero in precisely the following cases:
\begin{itemize}
\item $\lambda=(n)$ and $r\in\{1,2,\ldots, n-1\}$;

\item $\lambda=(1^n)$ and 
$r\in \begin{cases} \{1,2,\ldots, n-1\}, &n \text{ odd},\\
              \{1,2,\ldots, n-1, n\}\setminus\{\tfrac{n}{2}\},  &n \text{ even};
\end{cases}$
\item $\lambda=(n-1,1)$ and $r=n$;

\item $\lambda=(2, 1^{n-2})$  and
 $r=\begin{cases} n, & n\text{ odd},\\ \tfrac{n}{2}, & n\text{ even}; \end{cases}$

\item $\lambda=(2,2)$ and $r=1,3$,  $\lambda=(2,2,2)$ and $r=1,5$,  $\lambda=(3,3)$  and  $r=2,4$.
\end{itemize}
\end{theorem}

Finally, for the purposes of this paper, we make the following definition.
\begin{df}\label{defn:Schur-pos} Define a symmetric function $f$ of homogeneous degree $n$ to be Schur positive if it is the Frobenius characteristic of a true $\mathfrak{S}_n$-module. Equivalently,  $f$ is Schur positive if and only if it is a nonnegative integer combination of Schur functions.
\end{df}

\section{The sum $\sum_{d|n} p_d^{\tfrac{n}{d}}$}\label{sec:Rectangular}

In this section we investigate the subject of Chapoton's conjecture, the symmetric function $\sum_{d|n} p_d^{\tfrac{n}{d}}$.

The invertibility of the Ramanujan matrix $M_n$ and Proposition~\ref{prop:Orthogonality1}   imply the following. 

\begin{prop}\label{prop:basis} Fix $n\ge 1$. The symmetric  functions $\{\ell_n^{(m)}\}_{ m|n}$ form a Schur positive basis for the $\tau(n)$-dimensional subspace of the space of degree $n$ symmetric functions  spanned by the set $\{p_d^{\tfrac{n}{d}}\}_{d|n}.$

If $\{d_1,\ldots, d_{\tau(n)}\}$ is a fixed ordering of the divisors of $n$, then we have the matrix equation 
\begin{equation}\label{eqn:p-to-ell-matrix}
 \left[ p_{d_1}^{\tfrac{n}{d_1}}, \ldots,  p_{d_{\tau(n)}}^{\tfrac{n}{d_{\tau(n)}}} \right] 
=  \left[\ell_n^{\tfrac{n}{d_1}}, \ldots,  \ell_n^{\tfrac{n}{d_{\tau(n)}}}\right] \cdot M_n
\end{equation}
Hence for any scalars $\alpha(d)$, $d|n,$ we have 
\begin{equation}\label{eqn:lin-comb-p-to-ell}\sum_{d|n}\alpha(d) p_d^{\tfrac{n}{d}} 
= \sum_{k|n} \ell_n^{(\tfrac{n}{k})} \left( \sum_{d|n}c_k(\tfrac{n}{d}) \alpha(d)\right).\end{equation}
\end{prop}
\begin{proof} Theorem~\ref{thm:Foulkes} can be written as the matrix equation
\[\left[\ell_n^{\tfrac{n}{d_1}}, \ldots,  \ell_n^{\tfrac{n}{d_{\tau(n)}}}\right]
=\left[ p_{d_1}^{\tfrac{n}{d_1}}, \ldots,  p_{d_{\tau(n)}}^{\tfrac{n}{d_{\tau(n)}}} \right] \tfrac{1}{n} M_n,  \]
%
%
and the result follows from Proposition~\ref{prop:Orthogonality1}   by inverting the matrix $M_n$ and multiplying by the column vector whose $d$th entry is  $\alpha(d)$, recalling that the $(i,j)$ entry of $M_n$ is $c_{d_i}(\tfrac{n}{d_j})$. 
\end{proof}
The work of Section~\ref{sec:weightedRectangular}  will produce another basis of Schur positive symmetric functions for the space spanned by the set $\{ p_d^{\tfrac{n}{d}}\}_{d|n}$.

Recall from Definition~\ref{def:row-sum-Mn} that $a_k(n)=\sum_{d|n} c_k(\tfrac{n}{d})$ denotes the row sum of the matrix $M_n$ for the row corresponding to the divisor $k$ of $d$.

\begin{theorem}\label{thm:Chapoton-Shareshian}  We have the identity 
\begin{equation}\label{eqn:p-to-ell} \sum_{d|n} p_d^{\tfrac{n}{d}}=\sum_{k|n} a_k(n)\ell_n^{\tfrac{n}{k}}.\end{equation}
Moreover, the sum $\sum_{d|n} p_d^{\tfrac{n}{d}}$ is Schur positive.
\end{theorem}
\begin{proof}  
 Eqn.~\eqref{eqn:p-to-ell} is immediate upon setting $\alpha(d)=1$ for all $d|n$ in~ \eqref{eqn:lin-comb-p-to-ell}.
But Corollary~\ref{cor:Ramrowsum-nonneg} tells us that $a_k(n)\ge 0$ for all $k|d$, and hence the Schur positivity statement follows.
\end{proof}
Denote by $\Phi_n$ the representation whose Frobenius characteristic is 
$ \sum_{d|n} p_d^{\tfrac{n}{d}}.$

\begin{cor}\label{cor:restrict} The restriction to $\mathfrak{S}_{n-1}$ of $\Phi_n$ consists of $n$ copies of the regular representation. We then have, for the row sums $a_d(n)$ of the Ramanujan matrix, 
\[\sum_{d|n} a_d(n) =n=\sum_{d_1, d_2|n} c_{d_1}(\tfrac{n}{d_2}).\]
\end{cor}
\begin{proof}  This follows from~\eqref{eqn:p-to-ell}, since the representations whose characteristics are $\ell_n^{(r)}$ all restrict to the regular representation of $S_{n-1}.$ The fact that the sum of all the elements in the Ramanujan matrix equals $n$ also follows as a special case of the orthogonality relations, see 
\cite[Theorem 2.8 and p.79]{McCarthy1986}.
\end{proof}
\begin{cor}  The multiplicity of the Schur function $s_\lambda$ in $\sum_{d|n} p_d^{\tfrac{n}{d}}$ is the nonnegative integer equal to $\sum_{k|n} a_{\frac{n}{k}}(n) t_{\lambda, k},$ 
where $t_{\lambda, k}$ is the number of standard Young tableaux of shape  $\lambda$ with major index congruent to $k \bmod n$.  

Letting $\chi^\lambda$ be the character of the $\mathfrak{S}_n$-irreducible indexed by $\lambda$,  this quantity is also  
the sum of character values $\sum_{d|n} \chi^\lambda(d^{\tfrac{n}{d}}) $.
\end{cor}
\begin{proof}  
The result follows from Theorem~\ref{thm:Foulkes}, since the multiplicity of the Schur function $s_\lambda$ in $\sum_{d|n} p_d^{\tfrac{n}{d}}$ equals 
\begin{equation}
\sum_{k|n} a_{\frac{n}{k}}(n) \langle \ell_n^{(k)}, s_\lambda\rangle. \qedhere
\end{equation}
%
\end{proof}
We have the following special cases. 
As noted in \cite{Chapoton2021}, the sequence of multiplicities in Item (2) below  is OEIS A112329. Here we give an independent calculation using Theorem~\ref{thm:FowlerGarciaKaraali} and Theorem~\ref{thm:Swanson}.

\begin{prop}\label{prop:triv-sign-Rn0} For the representation $\Phi_n$ with Frobenius characteristic 
$\sum_{d|n} p_d^{\tfrac{n}{d}}$, 
\begin{enumerate}
\item  the trivial representation occurs $\tau(n)$ times;

\item The sign representation occurs with  multiplicity $\langle s_{(1^n)}, \Phi_n \rangle$ equal to 
\[\sum_{d|n} (-1)^{n+d}\!=\begin{cases} 
\tau(n),  &\text{ if } $n$ \text{ is odd},\\
\tau(\tfrac{n}{4}), &\text{if } 4 | n,\\
0, \!& \text{if $n$ is twice an odd integer}. 
\end{cases}\]
\item  The irreducible indexed by   $(n-1,1)$ occurs $n-\tau(n)$ times;

\item The irreducible indexed by $(2, 1^{n-1})$ occurs with multiplicity equal to \\ 
$n-\sum_{d|n} (-1)^{n+d}$.
\end{enumerate}
\end{prop}
\begin{proof}
 Item (1) is clear from the sum of power sums in the left-hand side of~\eqref{eqn:p-to-ell}.  

For Item (2), we have $\chi^{(1^n)}(d^{\tfrac{n}{d}})= (-1)^{(d-1)\tfrac{n}{d}}$, giving $\sum_{d|n}\chi^{(1^n)}(d^{\tfrac{n}{d}})=\sum_{d|n}(-1)^{n-\tfrac{n}{d}} =\sum_{d|n}(-1)^{n+d}.$ Now we use Theorem~\ref{thm:Swanson}, which asserts that the sign representation occurs with nonzero multiplicity in $\ell_n^{\tfrac{n}{k}}$ only if $k=1$ and $n$ is odd, or $k=2$ and $n$ is even. In either case, the multiplicity can be verified to be exactly one, from Part (2) of Theorem~\ref{thm:Foulkes}, for example, since the unique standard tableau of shape $(1^n)$ has major index  $n(n-1)/2.$ 
Now use~\eqref{eqn:p-to-ell} and Corollary~\ref{cor:Ramrowsum-nonneg}.  If $n$ is odd, the multiplicity is $\sum_{k|n} c_1(k) =\tau(n)$. 

If $n$ is even, the multiplicity equals the row sum corresponding to the divisor 2, and from the proof of Theorem~\ref{thm:FowlerGarciaKaraali}, this is 
\[f(n,2,2) \prod_{q>2} f(n,2,q).\]
Let  $n=2^\alpha m$, $m$ odd. Then   $f(n,2,2)=(\alpha-1+1)\phi(2) -\lfloor 1\rfloor=\alpha-1$.
 In particular this is zero if $n$ is twice an odd number. 

For primes $q\ne 2$ which divide $n$, say $n=q^i m$, $(q,m)=1$, we have, since $(q,2)=1$, 
$f(n,2,q)=(i +1)\phi(2) =i +1.$ 
It follows that $\prod_{q>2} f(n,2,q)=\prod_{q\ne 2} (i+1)=\prod_{q\ne 2} \tau(q^i)
=\tau(\tfrac{n}{2^\alpha})$.  Finally observe that $(\alpha-1)\tau(\tfrac{n}{2^\alpha})=\tau(\tfrac{n}{4})$ if $\alpha\ge 2$. 

 Items (3) and (4) follow by using Frobenius Reciprocity and computing the multiplicity of the trivial  (respectively the sign) representation in the restriction of $\Phi_n$ to $S_{n-1}$, which, from Corollary~\ref{cor:restrict},  consists of $n$ copies of the regular representation.
\end{proof}

\begin{prop}\label{prop:all-irreps} The representation $\Phi_n$ contains all the irreducibles except when $n$ is twice an odd number, in which case it contains all the irreducibles except for the sign.
\end{prop}
\begin{proof}  Fix a partition $\lambda$ of $n$. The multiplicity of the $\mathfrak{S}_n$-irreducible indexed by $\lambda$ is $\sum_{k|n} a_{\frac{n}{k}}(n) \langle \ell_n^{(k)}, s_\lambda\rangle. $ By Corollary~\ref{cor:Ramrowsum-nonneg}, $a_{\frac{n}{k}}(n)$ is strictly positive except when $n$ is even and $k$ is odd. 

By Proposition~\ref{prop:triv-sign-Rn0}   and Theorem~\ref{thm:Swanson}, we need only check that the multiplicity is positive for the following three cases: \mbox{$\lambda=(2,2), (2,2,2), (3,3)$}.

For $n=4$, we have $\ch\,\Phi_4= a_1(4) \ell_4^{(4)} + a_2(4)\ell_4^{(2)}$ since $a_4(4)=0$, and we know $a_1(4)=\tau(4)>0, a_2(4)>0$.   By Theorem~\ref{thm:Swanson}, the irreducible indexed by $(2,2)$ appears in $\ell_4^{(4)}$ and $\ell_4^{(2)}$, so we are done.

For $n=6$, we have $\ch\,\Phi_6= a_1(6) \ell_6^{(6)} + a_3(6) \ell_6^{(2)}$ since $a_2(6)=0=a_6(6)$.  Again 
the irreducible indexed by $(2,2,2)$ appears in both $\ell_6^{(6)}$ and $\ell_6^{(2)}$, while 
$(3,3)$ appears in $\ell_6^{(6)}$, so we are done.
\end{proof}

\begin{rem}  Note that $\frac{1}{n}\ch\, \Phi_n$= $\frac{1}{n} \sum_{d|n} p_d^{\frac{n}{d}}$ is not in general the Frobenius characteristic of a true $\mathfrak{S}_n$-module (cf. the Foulkes representations).  
We have the following decompositions into Schur functions for $\ch\,\Phi_n$:
\begin{align*}  \ch\,\Phi_2&= 2 s_{(2)} \\
                        \ch\,\Phi_3&=2 s_{(3)}+ s_{(2,1)} + 2 s_{(1^3)} \\
                        \ch\,\Phi_4&= 3 s_{(4)} + s_{(3,1)} + 4 s_{(2^2)} + 3 s_{(2,1^2)}+                           
                                          s_{(1^4)}\\
                        \ch\,\Phi_5&= 2 s_{(5)} + 3 s_{(4,1)}  +  5 s_{(3,2)} + 7 s_{(3,1^2)} + 5 s_{(2^2,1)} + 3 s_{(2,1^3)} + 2 s_{(1^5)} \\
                     \ch\, \Phi_6&= 4 s_{(6)} + 2 s_{(5,1)} + 12 s_{(4,2)} + 10 s_{(4,1^2)} + 4 s_{(3^2)} + 14 s_{(3,2,1)} + 12 s_{(3,1^3)} + 10 s_{(2^3)} + 6 s_{(2^2,1^2)} + 6 s_{(2, 1^4)}  
\end{align*}
\end{rem}

\section{The sum $\sum_{d|n} c_d(\tfrac{n}{d}) \, p_d^{\tfrac{n}{d}}$}\label{sec:weightedRectangular}

In this section we will prove the following variation of Theorem~\ref{thm:Chapoton-Shareshian}.  Recall from Definition~\ref{defn:Schur-pos} that we have defined a symmetric function to be Schur positive if it is a nonnegative integer combination of Schur functions.
\begin{theorem}\label{thm:DiagRamSum}  The symmetric function 
\begin{equation} R_n:=\sum_{d|n} c_d (\tfrac{n}{d}) p_d^{\tfrac{n}{d}}
\end{equation}
is Schur positive.  Moreover, if $T_o$ and $T_e$ are disjoint sets of primes such that 
$n=n_o n_e$ where $n_o=\prod_{q\in T_o} q^{2 a_q+1}$, $n_e=\prod_{q\in T_e} q^{2 b_q}$ for nonnegative integers $a_q, b_q$, then the symmetric function 
\[\frac{1}{n_o\sqrt{n_e}} R_n\]
is also Schur positive.
\end{theorem}

The  following special case of  Theorem~\ref{thm:DiagRamSum} follows immediately from Proposition~\ref{prop:Ramsum0pm1} (see Section~\ref{sec:Foulkes-reps} for the definition of $Lie_n$):
\begin{equation}\label{eqn:n-is-SqFree}
\text{When  $n$ is square-free,}\  R_n=n\cdot \ell_n^{(1)}=n Lie_n.
\end{equation}
Using Proposition~\ref{prop:basis} and Eqn.~\eqref{eqn:lin-comb-p-to-ell} with $\alpha(d)=c_d(\tfrac{n}{d})$, we see that 
$R_n=\sum_{k|n} Y[n,k] \ell^{(\tfrac{n}{k})},$ where 
\begin{equation}\label{eq:KeySum} Y[n,k]:=\sum_{d|n} c_k(\tfrac{n}{d}) c_d(\tfrac{n}{d}), \ k|n.
\end{equation}

Theorem~\ref{thm:DiagRamSum} will  follow if we can show that 
\begin{equation}\label{eq:DiagRam} Y[n,k] \text{ is a nonnegative integer for all } k|n, n\ge 1.
\end{equation}

Note that $Y[1,k]=Y[1,1]=1;$  also, from Proposition~\ref{prop:trM_n}, 

\[Y[n,1]=\begin{cases} \sqrt{n}, \text{ if $n$ is a perfect square}, \\
                                                                       0, \text{ otherwise}. \end{cases}\]
Moreover, if $n=q,$ $q$ prime, so that $k=1,q$, then 
\[Y[n,n]=Y[q,q]=c_q(q) c_1(q)+ c_q(1) c_q(1)=(q-1)+ 1=q. \]

The remainder of this section is devoted to establishing \eqref{eq:DiagRam}.

Note that 
\begin{equation}\label{eq:KeyExpansion}\sum_{k|n} Y[n,\tfrac{n}{k}] \ell_n^{(k)}=\sum_{d|n} c_d(\tfrac{n}{d}) p_d^{\tfrac{n}{d}}.\end{equation}

First we show that the multiplicative property (Corollary~\ref{cor:Factorisation})  of the  Ramanujan sum is inherited by the integers $Y[n,k]$:

\begin{lemma}\label{lem:multiplicativity}  Let $n=n_1 n_2,$ where $n_1, n_2$  are relatively prime.  Then 
every divisor $k$ of $n$ factors uniquely into $k=k_1k_2$ where $k_i|n_i, i=1,2,$
 and
\begin{equation}\label{eq:multiplicativity} Y[n,k] = Y[n_1, k_1] \cdot Y[n_2,k_2]. \end{equation}
\end{lemma}

\begin{proof} Each divisor $d$ of $n$ can be written uniquely as $d=d_1d_2$ with each $d_i$ a divisor of $n_i$.
Hence  we have 
\[Y[n,k]= \sum_{d_i|n_i, i=1,2} c_{k_1k_2}(\tfrac{n_1n_2}{d_1d_2})\  c_{d_1d_2}(\tfrac{n_1n_2}{d_1d_2}).\]
Apply Corollary~\ref{cor:Factorisation} to factor each of the two Ramanujan sums in each summand  above:
\[ c_{k_1k_2}(\tfrac{n_1n_2}{d_1d_2})=c_{k_1}(\tfrac{n_1}{d_1})\, c_{k_2}(\tfrac{n_2}{d_2});\]
\[ c_{d_1d_2}(\tfrac{n_1n_2}{d_1d_2})=c_{d_1}(\tfrac{n_1}{d_1})\, c_{d_2}(\tfrac{n_2}{d_2}).\]
It follows that 
\[Y[n,k] = \left(\sum_{d_1|n_1} c_{k_1}(\tfrac{n_1}{d_1})\,c_{d_1}(\tfrac{n_1}{d_1})\right)\cdot 
 \left(\sum_{d_2|n_2} c_{k_2}(\tfrac{n_2}{d_2})\,c_{d_2}(\tfrac{n_2}{d_2})\right),\]
as claimed.
\end{proof}
If $\{q_i\}_{i=1}^r$ are distinct primes and $S$ is a subset of $[r]:=\{1,2,\ldots, r\}$, set $q_S:=\prod_{i\in S} q_i,$ so that in particular 
$q_{[r]}:=\prod_{i=1}^r q_i.$  Clearly $q_S|q_T \iff S\subseteq T.$ 

\begin{lemma}\label{lem:n-has-SqFreepart} Let $n=x\prod_{i=1}^r q_i$ where $r\ge 1,$ $\{q_i\}_{i=1}^r$ are distinct primes, and $x$ is an integer relatively prime to $q_i$ for all $i.$  
Then every divisor $k$ of $n$ is of the form $k=zq_T$ for a divisor $z$ of $x$ and $T\subseteq [r],$ and we have 
\begin{equation}\label{eq:n-has-SqFreepart} Y[n,k]=Y[n, zq_T] =Y[x,z]\,  q_{[r]} \, \delta_{T, [r]}.
\end{equation}
\end{lemma}

\begin{proof}  From Lemma~\ref{lem:multiplicativity} we have 
\[Y[n, zq_T] =Y[x,z]\, Y[q_{[r]}, q_T].\]

But \[Y[q_{[r]}, q_T]
=\sum_{S\subseteq [r]}  c_{q_T}(\tfrac{q_{[r]}}{q_S})c_{q_S}(\tfrac{q_{[r]}}{q_S})
=  \sum_{S\subseteq [r]}  c_{q_T}(\tfrac{q_{[r]}}{q_S})\,\mu(q_S),\]
where the last reduction is a consequence of the special case~\eqref{eqn:SpecCasephimu} 
since 
$q_S$ and $\tfrac{q_{[r]}}{q_S}$ are relatively prime.  Finally the last sum equals $q_{[r]} \, \delta_{T, [r]}$ by Proposition~\ref{prop:KeyFact}.  \end{proof}

\begin{lemma}\label{lem:n-odd-prime-power} Let $q$ be prime and let $n=q^{2r+1}, r\ge 0.$ Then \[Y[n, k]=\begin{cases} n, \ k=q^{r+1},\\
                                             0, \ \text{otherwise}. \end{cases}\]
\end{lemma}

\begin{proof} The case $n=q$ has already been confirmed.  Observe the following fact: 
\begin{equation}\label{eq:sum-of-phi}\sum_{a=0}^r\phi(q^a)=q^r.\end{equation}

 Let $k=q^i, 0\le i\le 2r+1.$ Then 
\begin{equation}\label{eq:sumYnk}Y[n,k]= \sum_{a=0}^{2r+1} c_{q^i}(q^{2r+1-a}) c_{q^a}(q^{2r+1-a}).
\end{equation}
By Corollary~\ref{cor:prime-powers}, the second Ramanujan sum $c_{q^a}(q^{2r+1-a})$ in each summand on the right side of ~\eqref{eq:sumYnk}
is nonzero only if $a\le 2r+2-a$, i.e., only if $a\le r+1,$ so we in fact have 
\begin{equation}\label{temp1}Y[n,q^i]= \sum_{a=0}^{r+1} c_{q^i}(q^{2r+1-a}) c_{q^a}(q^{2r+1-a})
=\sum_{a=0}^{r} c_{q^i}(q^{2r+1-a}) \phi(q^a) +  c_{q^i}(q^{r}) (-q^r),
\end{equation}
again by  Corollary~\ref{cor:prime-powers}.

Let $i=r+1.$ Then $i=r+1\le 2r+1-a$ for all $0\le a\le r,$ so $ c_{q^i}(q^{2r+1-a})=\phi(q^i)$ and  \eqref{temp1} gives
\begin{align*} Y[n,q^{r+1}]&= \sum_{a=0}^r \phi(q^i) \phi(q^a) -q^r (-q^{i-1})\\
&=\phi(q^i) \sum_{a=0}^r  \phi(q^a) +q^{2r}\\
&=q^r(q-1) q^r+q^{2r}=q^{2r+1}=n,
\end{align*}
where we have used \eqref{eq:sum-of-phi} in the last line.

Now assume $i\le r.$ The least value of $2r+1-a,$  for $a=0,1,\ldots, r,$ is $r+1$
and this is greater than $i$ for any $i\le r,$ so from Corollary~\ref{cor:prime-powers}, equation \eqref{temp1} becomes
\[Y[n, q^i]=\sum_{a=0}^{r} \phi(q^i) \phi(q^a) -q^r\phi(q^i) =
\phi(q^i) [\sum_{a=0}^r \phi(q^a) -q^r]=0.
\]

Finally let $i\ge r+2.$ Then $c_{q^i}(q^r)=0$ and 
\[c_{q^i}(q^{2r+1-a})\ne 0\iff i\le 2r+2-a\iff a\le 2r+2-i.\]
But $i\ge r+2$ implies $2r+2-i\le r,$ so \eqref{temp1} becomes
 \begin{align*} Y[n, q^i]&=\sum_{a=0}^{2r+2-i} c_{q^i}(q^{2r+1-a}) \phi(q^a) \\
&=\sum_{a=0}^{2r+1-i} c_{q^i}(q^{2r+1-a}) \phi(q^a) -(q^{i-1} \phi(q^{2r+2-i}))\\
&=\sum_{a=0}^{2r+1-i} \phi(q^i)\phi(q^a) -(q^{i-1} \phi(q^{2r+2-i}))\\
&=\phi(q^i) q^{2r+1-i} -q^{i-1} q^{2r+1-i} (q-1) =0. \qedhere 
\end{align*}

\end{proof}

\begin{cor}\label{cor:Thm-n-odd-prime-power} Let $q$ be prime and let $n=q^{2r+1}, r\ge 0.$ Then the symmetric function $R_n$ of Theorem~\ref{thm:DiagRamSum} equals 
\[R_n= n\ell_n^{(q^r)}\] 
and is thus Schur positive.  Moreover the symmetric function $\frac{1}{n} R_n$ is also Schur positive.
\end{cor}

\begin{lemma}\label{lem:n-even-prime-power}Let $q$ be prime and let $n=q^{2r}, r\ge 1.$ Then \[Y[n, q^i]=\begin{cases} q^r \phi(q^i), \ 0\le i\le r,\\
                                             0, \ \text{otherwise}. \end{cases}\]
\end{lemma}

\begin{proof} 
We have \[Y[n, q^i]=\sum_{a=0}^{2r} c_{q^i}(q^{2r-a}) \ (c_{q^a}(q^{2r-a}) ).\]
The second Ramanujan sum in each summand is nonzero only if $a\le 2r-a+1,$ in which case $a\le r$ and $c_{q^a}(q^{2r-a}) =\phi(q^a),$ and the first Ramanujan sum is nonzero only if $i\le 2r-a+1.$ Hence we have 
 \begin{equation}\label{temp2}
Y[n, q^i]=\sum_{a=0}^{r} c_{q^i}(q^{2r-a}) \phi(q^a)
=\sum_{a=0}^{\min(r, 2r-i+1)} c_{q^i}(q^{2r-a}) \phi(q^a) .
\end{equation}
If $i\ge r+1,$ then $2r-i+1\le r$ and \eqref{temp2} becomes
\[Y[n, q^i]
=\sum_{a=0}^{2r-i+1} c_{q^i}(q^{2r-a}) \phi(q^a)
=\sum_{a=0}^{2r-i} \phi(q^i) \phi(q^a) +(-q^{i-1})\phi(q^{2r-i+1}) .\]
As before, this reduces to 
$\phi(q^i) q^{2r-i+1}-q^{i-1}(q-1) q^{2r-i+1}=0.$

Now let $0\le i\le r.$  Then \eqref{temp2} gives
\[Y[n, q^i]=\sum_{a=0}^{r} c_{q^i}(q^{2r-a}) \phi(q^a)=\sum_{a=0}^{r} \phi(q^i) \phi(q^a)
=\phi(q^i) q^r,\]
again using Corollary~\ref{cor:prime-powers} and Equation \eqref{eq:sum-of-phi}.
\end{proof}

Hence we have:
\begin{cor}\label{cor:Thm-n-even-prime-power} Let $q$ be prime and let $n=q^{2r}, r\ge 1.$ Then the symmetric function $R_n$ of Theorem~\ref{thm:DiagRamSum} equals 
\[R_n=q^r  \sum_{i=0}^r \phi(q^i)\ell_n^{(q^{2r-i})}\] 
and is thus Schur positive.  Moreover the symmetric function $\frac{1}{\sqrt{n}} R_n$ is also Schur positive.
\end{cor}

\begin{proof} We have 
\[R_n=\sum_{i=0}^{2r}  Y[n, q^i] \ell_n^{(q^{2r-i})}=\sum_{i=0}^r q^r \phi(q^i) \ell_n^{(q^{2r-i})}. \qedhere\]
\end{proof}

\begin{proof}[Proof of Theorem~\ref{thm:DiagRamSum}]  The case when $n$ is square-free was already observed in Eqn.~\eqref{eqn:n-is-SqFree}.  The general case follows from Lemmas~\ref{lem:multiplicativity}, \ref{lem:n-odd-prime-power} and ~\ref{lem:n-even-prime-power}, showing recursively that the $Y[n,k]$ are nonnegative integers for all divisors $k$ of $n.$ 
Hence the function $R_n$ is always a nonegative integer combination of the $\{\ell_n^{(m)}\}_{m|n},$  and is therefore Schur positive.  The last statement in Theorem~\ref{thm:DiagRamSum}, asserting that $R_n$ is a positive integer multiple of a Schur positive function,  follows from Eqn.~\eqref{eqn:n-is-SqFree} and the corresponding statements in  
Corollary~\ref{cor:Thm-n-odd-prime-power} and Corollary~\ref{cor:Thm-n-even-prime-power}.  \end{proof}

\begin{prop}\label{prop:mult} For the representation with Frobenius characteristic $R_n:$
\begin{enumerate}
\item The restriction to $\mathfrak{S}_{n-1}$ equals $n$ copies of the regular representation.
\item If $n$ is odd, it is self-conjugate.
\item The multiplicity of the trivial representation is $\sum_{d|n} c_d(\tfrac{n}{d}).$ It is nonzero only if $n$ is a perfect square, in which case it equals $\sqrt{n}.$  
\item The multiplicity of the sign representation equals that of the trivial representation if $n$ is odd, and if $n$ is even, it 
equals
$\left\{ \begin{array}{ll} \sqrt{n}, & n \mbox{ a perfect square,} \\ 2 \sqrt{n/2}, & n/2 \mbox{ an odd perfect square,} \\ 0, & \mbox{otherwise.} \end{array} \right.$
\item The multiplicity of the irreducible indexed by $(n-1,1)$ (respectively $(2, 1^{n-2})$) is $n$ minus the multiplicity of the trivial (respectively the sign) representation.
\end{enumerate}
\end{prop}
\begin{proof} The first two parts are clear. Items (3) and (4) follow from Proposition~\ref{prop:trM_n} and Proposition~\ref{prop:signed-trM_n} respectively.  The last statement is a consequence of Item (1) and Frobenius reciprocity.
\end{proof}

\section{The sums $\sum_{d|n} (c_d(\tfrac{n}{d}))^u p_d^{\tfrac{n}{d}}, \ \text{arbitrary }u\ge 1$}

Let $u$ be a nonnegative integer and define 

\[R_{n,u}:=\sum_{d|n} (c_d(\tfrac{n}{d}))^u p_d^{\tfrac{n}{d}},\]
and 
\begin{equation}\label{eq:KeySum-u}Y_u[n,k]:=\sum_{d|n} c_k( \tfrac{n}{d} )(c_d(\tfrac{n}{d}))^u, \, k|n.\end{equation}

Thus $R_{n,0}$ is the symmetric function $\Phi_n$ of Section~\ref{sec:Rectangular}, and $R_{n,1}$ is the symmetric function $R_n$ of Section~\ref{sec:weightedRectangular}.

As before,
\[R_{n,u}=\sum_{k|n} Y_u[n,\tfrac{n}{k}]\, \ell_n^{(k)},\]
and hence, if all the $Y_u[n,k]$ are nonnegative integers,  the symmetric function $R_{n,u}$ will be Schur positive.

 Lemmas~\ref{lem:multiplicativity} and~\ref{lem:n-has-SqFreepart} 
 generalise in a straightforward manner to the $Y_u$ for arbitrary nonnegative integers $u$, as follows.

Again we show that the multiplicative property Corollary~\ref{cor:Factorisation}  of the  Ramanujan sum is inherited by the integers $Y_u[n,k]$. Exactly as in  Lemma~\ref{lem:multiplicativity}, we have the following.

\begin{lemma}\label{lem:multiplicativity-u}  Let $n=n_1 n_2,$ where $n_1, n_2$  are relatively prime.  Then 
every divisor $k$ of $n$ factors uniquely into $k=k_1k_2$ where $k_i|x_i, i=1,2,$
 and
\begin{equation}\label{eq:multiplicativity-u} Y_u[n,k] = Y_u[n_1, k_1] \cdot Y_u[n_2,k_2]. \end{equation}
\end{lemma}
%

\begin{lemma}\label{lem:n-has-SqFreepart-u} Let $n=x\prod_{i=1}^r q_i$ where $r\ge 1,$ $\{q_i\}_{i=1}^r$ are distinct primes, and $x$ is an integer relatively prime to $q_i$ for all $i.$  
Then every divisor $k$ of $n$ is of the form $k=zq_T$ for a divisor $z$ of $x$ and $T\subseteq [r],$ and we have 
\begin{equation}\label{eq:n-has-SqFreepart-u} Y_u[n,k]=Y_u[n, zq_T] =Y_u[x,z]\cdot
\begin{cases} q_{[r]} \, \delta_{T, [r]}, u \ \mathrm{odd},\\
                      \sum_{S\subseteq [r]}  c_{q_T}(\tfrac{q_{[r]}}{q_S}), u \ \mathrm{even}.
\end{cases}
\end{equation}
In particular, $Y_u[n, zq_T]$ is always a nonnegative integer multiple of $Y_u[x,z].$ 
\end{lemma}

\begin{proof}  From Lemma~\ref{lem:multiplicativity-u} we have 
\[Y_u[n, zq_T] =Y_u[x,z]\, Y_u[q_{[r]}, q_T].\]

But \[Y_u[q_{[r]}, q_T]
=\sum_{S\subseteq [r]}  c_{q_T}(\tfrac{q_{[r]}}{q_S}) (c_{q_S}(\tfrac{q_{[r]}}{q_S}))^u
=  \sum_{S\subseteq [r]}  c_{q_T}(\tfrac{q_{[r]}}{q_S})\,(\mu(q_S))^u,\]
where the last reduction is a consequence of the special case~\eqref{eqn:SpecCasephimu} 
since 
$q_S$ and $\tfrac{q_{[r]}}{q_S}$ are relatively prime.  

When $u$ is odd, $(\mu(q_S))^u=\mu(q_S),$ and so the last sum equals $q_{[r]} \, \delta_{T, [r]}$ by Proposition~\ref{prop:KeyFact}.

When $u$ is even, the last sum equals $\sum_{S\subseteq [r]}  c_{q_T}(\tfrac{q_{[r]}}{q_S}),$ and this 
 is a row sum of the matrix $M_{q_{[r]}};$ it is therefore nonnegative by Theorem~\ref{thm:FowlerGarciaKaraali}.  
\end{proof}

As in the previous section, Proposition~\ref{prop:Ramsum0pm1} gives us the following, which also follows from the above calculations of $Y_u[n,k]$.
\begin{prop}\label{prop:n-is-SqFree-4m-u}  Let $u$ be a nonnegative integer.  

First let $u$ be odd.  If $n$ is square-free or 4 times an odd square-free integer, then 
  the symmetric function $R_{n,u}=\sum_{d|n} (c_d (\tfrac{n}{d}))^u p_d^{\tfrac{n}{d}}$ coincides with 
\[R_n=\sum_{d|n} c_d(\tfrac{n}{d}) p_d^{\tfrac{n}{d}}\]
and is thus Schur positive. 

Now suppose $u$ is even.   Then 
\[R_{n,u}=\begin{cases} R_{n,0}=\sum_{d|n} p_d^{\tfrac{n}{d}}, 
                                                                                & n\ \text{square-free},\\
                         R_{n,2} =\sum_{d|n}  (c_d(\tfrac{n}{d}))^2 p_d^{\tfrac{n}{d}}, 
& n=4m,  \text{where $m$ is odd and square-free}.
\end{cases}
\]
In particular, for any $u\ge 0$, $R_{n,u}$ is 
  a nonnegative linear combination of $\{\ell_n^{(d)}\}_{d|n}$ when $n$ is square-free. 
\end{prop}

\begin{cor}\label{cor:powers-of-4-u} Let $n=4m,$ where $m$ is square-free and odd. Then 
$R_{n,u}$ is a nonegative linear combination of $\{\ell_n^{(d)}\}_{d|n}$ and hence it is Schur positive.
\end{cor}

\begin{proof} We claim that $Y_u[n, k] \ge 0$ for every divisor $k$ of $n,$ 
for all $u\ge 1.$ This follows from 
 $Y_u[4,k]=c_k(4) + c_k(2) (c_2(2))^u+ c_k(1) (c_4(1))^u
=c_k(4) +c_k(2)$ for all $u\ge 0.$  Hence $Y_u[4,k]=1$ if $k=1,2$ and $Y_u[4,4]=2$, 
so $Y_u[4,k]$ is positive for all $u$ and $k|4$. 
Now apply the reduction of Lemma~\ref{lem:n-has-SqFreepart-u} and  Proposition~\ref{prop:n-is-SqFree-4m-u}.
\end{proof}

We can now summarise our results as follows.
\begin{theorem}\label{thm:DiagRamSumPowers} Let $u$ be a nonnegative integer and $n$ a positive integer.  The symmetric function $R_{n,u}$ is a positive integer combination of 
$\{\ell_n^{(m)}\}_{m|n}$ in the following cases:
\begin{enumerate}
\item $u=0,1$  and $n\ge 1;$
\item  $u\ge 0$ and $n$ is square-free; 
\item $u\ge 0$ and  $n=4 m$ where $m$ is odd and square-free.
\end{enumerate} 
\end{theorem}
From  the preceding theorem and Proposition~\ref{prop:Ramsum0pm1}, 
we see that for $u\ge 2$, it remains to consider the values of $n$ for which  
 there is a divisor $d$ such that $|c_d(\tfrac{n}{d})|>1$.  Let $\chi_{n,u}$ denote the character of the (possibly virtual) representation whose Frobenius characteristic is $R_{n,u}$.  Then 
the character value on the conjugacy class indexed by  the partition $d^{\tfrac{n}{d}}$ equals 
\[\chi_{n,u}(d^{\tfrac{n}{d}}) = (c_d(\tfrac{n}{d}))^u d^{\tfrac{n}{d}} (\tfrac{n}{d})!\]
Since  $|c_d(\tfrac{n}{d})|>1$, there is always a $u$ sufficiently large such that $|\chi_{n,u}(d^{\tfrac{n}{d}})| >n!=\chi_{n,u}(1^n)$.   It follows that for this value of $u$, $\chi_{n,u}$ cannot be the character of a true $\mathfrak{S}_n$-module \cite[Lemma 2.15]{Isaacs1994}.  Hence for $u$ sufficiently large, there is always a value of $n$ such that $R_{n,u}$ is not Schur positive. 

 For example, quick calculations using the argument just given show that if $u\ge 15$ then $R_{8,u}$ is not Schur positive. It turns out (see Proposition~\ref{prop:prime-powers-u-NOT-pos} below) that $R_{8,u}$ is not Schur positive for $u\ge 3$.

Unfortunately,  the positivity results in the analogues of Lemmas~\ref{lem:n-odd-prime-power} and \ref{lem:n-even-prime-power} no longer hold for $u\ge 2$. 
\begin{ex}\label{nonnegLinCombElls} Data from  Maple and Stembridge's SF package gives the following:

$R_{8,2}$ and $R_{9,2}$ ARE Schur positive, but their expansions in the $\ell_n^{(m)}$ are not:
\[R_{8,2}=6 \ell_8^{(8)} + 6 \ell_8^{(4)} -4 \ell_8^{(2)},\]
\[R_{9,2}=5 \ell_9^{(9)} + 10 \ell_9^{(3)} -6 \ell_9^{(1)}.\]
\end{ex}

Although the functions $R_{n,u}$ are no longer a positive combination of $\{\ell_n^{(m)}\}_{m|n}$ 
for $u\ge 2,$  the data still supports the following conjecture for the case $u=2$:  
\begin{conj}\label{conj:Rn2}  The symmetric function $R_{n,u}=\sum_{d|n} (c_d(\tfrac{n}{d}))^u\, p_d^{\tfrac{n}{d}}$ is Schur positive for all $n$ when $u=2.$  
\end{conj}

Finally we consider the multiplicities of the trivial representation and the irreducible indexed by $(n-1,1)$ in the representation $R_{n,u}$. 
The multiplicity of the trivial representation is 
\begin{equation}\label{eqn:Rnu-triv} t(n,u)=\sum_{d|n} c_d(\tfrac{n}{d})^u,
\end{equation}
and the multiplicity of the irreducible indexed by $(n-1,1)$ is 
\begin{equation}\label{eqn:Rnu-refl} n-t(n,u)= n-\sum_{d|n} c_d(\tfrac{n}{d})^u,
\end{equation}
since rectangular partitions $(d^{\tfrac{n}{d}})$ have fixed points only for $d=1$.  The analogous statement holds for the multiplicity of the irreducible indexed by $(2, 1^{n-2})$, since the multiplicity of the sign representation is $\sum_{d|n} c_d(\tfrac{n}{d})^u (-1)^{n-\tfrac{n}{d}}$. 

Explicit computations of these multiplicities show that $\Phi_{n,u}$ is not Schur positive for certain $n$ and small $u$, thereby explaining some entries in Table~\ref{tab:u-Data} below.
\begin{prop}\label{prop:prime-powers-u-NOT-pos} Fix a prime $q$.  Let $u\ge 3$.
\begin{enumerate}
\item 
If $n=q^{2k+1}, k\ge 1$, then $R_{n,u}$ is never Schur positive.
\item
If $n=q^{2k}, k\ge 1$ then $R_{n,u}$ is never Schur positive when $u\ge 4$. If $u=3$, it is never Schur positive except when   $n=9, 16$. 
\item If $n$ is twice an odd prime power, then $R_{n,u}$ is never Schur positive for $u\ge 4$. It is never positive for $u=3$ unless $n=18$.   
\end{enumerate}
\end{prop}
\begin{proof} First note that a necessary condition for  Schur positivity of $R_{n,u}$ is that 
$0\le t(n,u)\le n$. 
Also note that for any $x\ge 2$, and any integer $u\ge 2$, we always have the elementary inequality 
\begin{equation}\label{eqn:basic-ineq} x^u-1>(x-1)^u. \end{equation}

 Let $n=q^{2k+1}$, an odd power of $q$.  Computing $t(n,u)$ using~\eqref{eqn:Rnu-triv} and Proposition~\ref{cor:prime-powers}, we obtain
\[ t(n,u)=1+(q^{ku}-1) \frac{(q-1)^u}{q^u-1} +(-1)^u q^{ku}.\]
When $u$ is odd,~\eqref{eqn:basic-ineq} gives
\[ t(n,u)=1+(q^{ku}-1) \frac{(q-1)^u}{q^u-1} - q^{ku}<0.\]
When $u$ is even, since $u\ge 3$ we have 
\[t(n,u)=1+(q^{ku}-1) \frac{(q-1)^u}{q^u-1} + q^{ku}>q^{ku}\ge q^{4k}>q^{2k+1}=n.\]
In either case, the necessary condition $0\le t(n,u)\le n$ fails, so $R_{n,u}=R_{q^{2k+1},u}$ is not Schur positive.

  Now let $n=q^{2k}, k\ge 2$ be an even power of $q$. One checks by direct computation 
(see Table~\ref{tab:u-Data}) that  Schur positivity holds for 
$R_{9,3}=p_1^9+8 p_3^3  $  and  for $R_{16,3}= p_1^{16}+p_2^4+ 8 p_4^4$.

In this case, using~\eqref{eqn:Rnu-triv} and Proposition~\ref{cor:prime-powers} shows that $t(n,u)>0$, since 
\[ t(n,u)=1+(q^{ku}-1) \frac{(q-1)^u}{q^u-1}.\] 
We claim that except for the cases stated above, $t(n,u)>q^{2k}$ for all $u\ge 3$. 

For $k=1$, $t(n,u)=1+(q-1)^u\ge 1+(q-1)^3,$  and $1+(q-1)^3>q^2$ if and only if $(q-1)^2>q+1$, which holds for all $q>3$.

For $k=2$, $t(n,u)=1+(q^{u}+1)(q-1)^{u}>q^4$ if and only if $(q^{u}+1)(q-1)^{u-1}>q^3+q^2+q+1$. But 
  \[(q^{u}+1)(q-1)^{u-1}\ge (q^{3}+1)(q-1)^2=(q^3+q^2+q+1)
+q^4(q-3)+q(q^3-3),\] and hence $t(n,u)> q^4$ for $q\ge 3$. 

Now let $k\ge 3$.  Expanding the geometric series in powers of $q^u$ shows that $\frac{q^{ku}-1}{q^u-1}\ge\frac{q^{3k}-1}{q^3-1}$ when $u\ge 3$. Hence 
\begin{align*} 
(t(n,u)-q^{2k})(q^3-1)
&\ge (q^{3k}-1) (q-1)^3 -(q^{2k}-1)(q^3-1)\\
&>q^{3k}-1 -q^{2k+3}+q^3+q^{2k}-1\\
&= q^{2k}(q^k-q^3+1) +(q^3-2)>0
\end{align*}
for $k\ge 3$, once again violating the necessary condition for Schur positivity. 

Finally let $n=2\, q^r$, where $q$ is an odd prime. 
Then $t(\tfrac{n}{2},u)$ is the multiplicity of the trivial representation in $R_{q^r,u}$. 
Corollary~\ref{cor:Factorisation} implies that 
\begin{equation*} \sum_{d|n} c_d(\tfrac{n}{d})^u
 = (1+(-1)^u) 
\left(\sum_{d|\tfrac{n}{2}} c_d(\tfrac{n/2}{d})^u +\sum_{d|\tfrac{n}{2}} c_d(\tfrac{n/2}{d})^u\right),
\end{equation*} and
\begin{equation*} \sum_{d|n} c_d(\tfrac{n}{d})^u (-1)^{n-\tfrac{n}{d}}
 = (1+(-1)^{u-1}) 
\left(\sum_{d|\tfrac{n}{2}} c_d(\tfrac{n/2}{d})^u +\sum_{d|\tfrac{n}{2}} c_d(\tfrac{n/2}{d})^u\right).
\end{equation*}
It follows  that 
$2\, t(\tfrac{n}{2},u)$ equals 
\begin{itemize}
\item
the multiplicity  of the trivial representation in $R_{n,u}$ when $u$ is even, and 
\item the multiplicity of the sign representation in $R_{n,u}$ when $u$ is odd.
\end{itemize}

The previous analysis of $t(n,u)$ now completes the proof.                     
\end{proof}

Conjecture~\ref{conj:Rn2} is confirmed by the data in Table~\ref{tab:u-Data} below, which was generated  using Stembridge's SF package for Maple, for $u=2$ and $n\le 45$. Note that, in view of Theorem~\ref{thm:DiagRamSumPowers}, the table excludes the cases when $n $ is square-free or 4 times an odd square-free number.  In fact all the failures in the table can be accounted for by checking the bounds on the multiplicity of the trivial or sign representations in $R_{n,u}$, as in Proposition~\ref{prop:prime-powers-u-NOT-pos}. 
Finally we note that unlike the case $u=1$, the symmetric functions $R_{n,2}$ do not appear to be integer multiples of Schur positive functions.  

\begin{table}[h]
\begin{center}
\begin{tabular}{|c|c|c|c|c|c|c|c|c|c|c|l|}
\hline
$u\backslash n$ & 8  &9 & 16 & 18 & 24 & 25 & 27 & 32 & 36 & 40 & 45 \\
[2pt]\hline
%
$ u=0, 1$ & {\bf Y} & {\bf Y} & {\bf Y} & {\bf Y} & {\bf Y} & {\bf Y} & {\bf Y} & {\bf Y} & {\bf Y}  &{\bf Y} & {\bf Y}\\
$2$ & {\bf Y} & {\bf Y} & {\bf Y} & {\bf Y} & {\bf Y} & {\bf Y} & {\bf Y} & {\bf Y} & {\bf Y}  &{\bf Y} &{\bf Y} \\
$3$ & N & {\bf Y} & {\bf Y} & {\bf Y} & {\bf Y} & N & N & N & {\bf Y}  &{\bf Y} 
& {\bf Y}\\
$4$ & N & N & N & N & N & N & N & N & {\bf Y}  &{\bf Y} &{\bf Y} \\
$5$ & N & N & N & N & {\bf Y} & N & N & N & N &{\bf Y} & {\bf Y}\\
$6$ & N & N & N & N & N & N & N & N & N &N &N \\
$7$ & N & N & N & N & {\bf Y} & N & N & N & N & {\bf Y} & {\bf Y}\\
$8$ & N & N & N & N & N & N & N & N & N &N &N \\
$9$ & N & N & N & N & {\bf Y} & N & N & N & N &{\bf Y} &{\bf Y} \\
$10$ & N & N & N & N & N & N & N & N & N &N & N \\
$11$ & N & N & N & N & N & N & N & N & N &{\bf Y} & N\\
$12$ & N & N & N & N & N & N & N & N & N &N & N\\
$13\le u \le 20$ & N & N & N & N & N & N & N &N & N &N & N\, $\scriptstyle{u=13}$ \\
\hline
\end{tabular}
\end{center}
\vskip .1in
\caption{\small Schur positivity of $R_{u,n}, 0\le u\le 20.$}
\label{tab:u-Data}
\end{table}

\noindent\textbf{Acknowledgement:} John Shareshian was supported by NSF grant DMS 1518389.

The authors are very grateful to the referee for a careful reading of the paper, and many valuable comments and suggestions for improved exposition.
%
%

\end{document}